\documentclass{amsart}
 \usepackage{amsthm,amsfonts,amsmath,amssymb,latexsym,epsfig,upref,eucal,ae}

\usepackage[all]{xy}
\usepackage{color}

\theoremstyle{plain}
\newtheorem{theorem}{Theorem}[section]
\newtheorem{lemma}[theorem]{Lemma}
\newtheorem{corollary}[theorem]{Corollary}
\newtheorem{proposition}[theorem]{Proposition}

\theoremstyle{definition}
\newtheorem{definition}[theorem]{Definition}
\newtheorem{definition-theorem}[theorem]{Definition-Theorem}

\theoremstyle{remark}
\newtheorem{remark}[theorem]{Remark}

\def\Aut{\mathrm{Aut}}

\def\trace{\mathrm{trace}}

\def\Id{\mathrm{Id}}

\def\z{\mathrm{z}}

\def\Lie{\mathrm{Lie}}
\def\Ad{\mathrm{Ad}}

\def\cF{\mathcal{F}}
\def\cO{\mathcal{O}}

\def\cG{\mathcal{G}}

\def\cB{\mathcal{B}}
\def\cL{\mathcal{L}}

\def\cZ{\mathcal{Z}}

\def\t{\mathfrak{t}}

\def\g{\mathfrak{g}}

\def\R{\mathbb{R}}

\def\P{\mathbb{P}}

\def\C{\mathbb{C}}

\def\bs{{\bf s}}

\def\Om{\Omega}
\def\om{\omega}

\def\>{\rangle}

\def\<{\langle}
\def\>{\rangle}

\def\Ch{\mathrm{Chow}}
\def\cCh{\mathcal{CHOW}}
\def\Sym{\mathrm{Sym}}

\def\tG{\widetilde G}
\def\tT{\widetilde T}
\def\tg{\tilde{\mathfrak{g}}}
\def\tt{\tilde{\mathfrak{t}}}

\begin{document}

\title[Chow stability and optimal weights]
{Relative Chow stability and optimal weights}
\author[C. Tipler]{Carl Tipler}
\address{Universit\'e de Bretagne Occidentale, 6, avenue Victor Le Gorgeu, 29238 Brest Cedex 3 France}
\email{carl.tipler@univ-brest.fr}

\date{\today}

\begin{abstract}
For a polarized K\"ahler manifold $(X,L)$, we show the equivalence between relative balanced embeddings introduced by Mabuchi
and $\sigma$-balanced embeddings introduced by Sano, answering a question of Hashimoto.
We give a GIT characterization of the existence of a $\sigma$-balanced embedding, 
and relate the optimal weight $\sigma$ to the action of $\Aut_0(X,L)$ on the Chow line of $(X,L)$.
\end{abstract}

\maketitle

\section{Introduction}
By definition \cite{c1}, an extremal K\"ahler metric
on a polarized K\"ahler manifold $(X,L)$ is a critical point of the Calabi functional, which
assigns to each K\"ahler metric in $c_1(L)$ the $L^2$-norm of its scalar curvature.
Constant scalar curvature K\"ahler metrics (cscK for short) are special examples of extremal metrics.
Initially stated for cscK metrics, the Yau-Tian-Donaldson conjecture, refined by Sz\'ekelyhidi \cite{sz},
predicts that the existence of an extremal K\"ahler metric on a given polarized K\"ahler manifold $(X,L)$ should be equivalent to a
relative GIT stability notion of $(X,L)$. This conjecture should be seen as an infinite dimensional Kempf-Ness correspondence.\\

In \cite{Don01}, Donaldson introduced a finite dimensional approximation to this Kempf-Ness correspondence, via quantization.
Let $\Aut_0(X,L)$ be the automophism group of $(X,L)$ modulo the $\C^*$-action by rotation on the fibers of $L$.
If this group is discrete, cscK metrics can be approximated by a sequence of balanced metrics. 
Balanced metrics are particular K\"ahler metrics induced by projective embeddings $X\hookrightarrow \P(H(X,L^k)^*)$.
They appear as zeros of a finite dimensional moment map, and their existence correspond to the Chow stability of $(X,L^k)$ \cite{zha,luo,PS03,ps04}.
From the general theory of moment maps, balanced metrics satisfy a unicity property.
As a corollary, one obtains uniqueness of a cscK metric in $c_1(L)$ under the assumption on $\Aut(X,L)$ \cite{Don01}.
The quantization method is a powerful tool in the study of cscK metrics \cite{Don01,Don05},
and it is natural to extend it to the situation $\Aut_0(X,L)\neq 0$.\\

Generalisation of this approximation process to extremal metrics has been pioneered by Mabuchi \cite{ma040,ma04-1,ma04,ma05}. 
His approach can be understood in the framework of relative stability \cite{sz,ah}, 
which naturally appears in GIT in the presence of non-discrete stabilizers.
In the quantization setting, elements of $\Aut_0(X,L)$ acting on the Chow line of $(X,L)$
are an obstruction to Chow stability \cite{ma04-1,ma05,futaki04}.
This obstruction is a source of examples of cscK manifolds that are not asymptotically Chow stable \cite{dz,osy}.
It is then natural to consider a stability notion relatively to a maximal
torus of symmetries $T\subset \Aut_0(X,L)$, so-called Chow polystability relative to $T$ \cite{ma04,ah}.\\

Since then, other notions of quantization of extremal metrics have been introduced.
Extremal metrics can be seen as self-similar solutions to the Calabi flow.
As noticed by Sano, it is also natural to consider Donaldson's iteration process \cite{Don01} as a discretisation
of the Calabi flow. The notion of $\sigma$-balanced metrics, for $\sigma\in\Aut_0(X,L)$,
provides a quantization of extremal metrics by self-similar solutions to Donalsdon's iteration process \cite{s}.
The $\sigma$-balanced metrics appear as zeros of a finite dimensional moment map, and can be used to recover
minimization properties, unicity or splitting results for extremal metrics \cite{st1,st2}. Another point of view is given by
Hashimoto, whose defined a quantization of the extremal vector field \cite{hashimoto1}.
We refer to \cite{hashimoto2} for a good reference on these different notions of quantization.\\

In this note, we show that Mabuchi's relatively balanced metrics and Sano's $\sigma$-balanced metrics are 
equivalent, for a suitable choice of $\sigma$. In fact, for a given maximal compact torus $T\subset \Aut_0(X,L)$,
there is a unique $\sigma \in T^c$ (modulo $T$) allowing the existence of a $\sigma$-balanced metric.
Such an automophism $\sigma$ will be called an optimal weight.

\begin{theorem}
 \label{theo:intro}
 Let $(X,L)$ be a polarized compact K\"ahler manifold and let $T\subset\Aut_0(X,L)$ be a maximal
 compact torus. Then $(X,L)$ is Chow polystable relative to $T$ if and only if it admits a $\sigma$-balanced
 embedding with $\sigma\in T^c$ optimal weight.
\end{theorem}

The paper is organized as follows. Section \ref{sec:Chow scheme} is a brief review on Chow stability, 
in relation to balanced embeddings. We introduce necessary notations, definitions and results.
In section \ref{sec:rel Chow}, we define the relative Chow scheme $\cCh_N^T(n,d)$. 
Restricting to the smooth case, $\cCh_N^T(n,d)$ parametrizes $T^c$-invariant subvarieties of $\C\P^N$ of degree $d$ and dimension $n$, 
for $T^c$ a given subtorus of $\Aut(\C\P^N)$. It provides the natural framework for relative Chow stability.
We then prove Theorem \ref{theo:intro} in Section \ref{sec:GIT sigma}. Along the way, we explain how an optimal weight $\sigma$
twists the torus of symmetries $T$ to another torus $T^\sigma$ so that the Chow line of $(X,L)$ becomes $T^\sigma$-invariant.
Finally, we give a closed formula relating the optimal weight $\sigma$
to the $T$-character acting on the Chow line of $(X,L)$ in Section \ref{sec:optimal weight}.

\subsection{Acknowledgments} The author is grateful to Yuji Sano for sharing his ideas
on $\sigma$-balanced metrics. He would like to thank Vestislav Apostolov and Yoshinori Hashimoto for stimulating discussions,
as well as CIRGET and IMPA for their hospitality. 
He benefits from the supports of the French government ``Investissements d'Avenir'' program ANR--11--LABX--0020--01,
and ANR project EMARKS No ANR--14--CE25--0010. 

\section{Chow stability and balanced embeddings}
\label{sec:Chow scheme}
\noindent Let $(X,L)$ be a polarized K\"ahler manifold of complex dimension $n$.
By Kodaira's embedding theorem,
replacing $L$ by a sufficiently large tensor power, the following map defines an embedding:
$$
\begin{array}{cccc}
 \iota : & X & \rightarrow & \P(V^*) \\
  & x & \mapsto & [ev_x],
\end{array}
$$
where $ev_x$ denotes the evaluation map at $x\in X$ and $V:=H^0(X,L)$.
Let $\cB(V)$ be the space of basis of $V$.
For any basis $\bs=\lbrace s_\alpha\rbrace\in \cB(V)$ we define an isomorphism
\begin{equation}
 \label{eq:isomPhis}
\begin{array}{cccc}
 \Phi_\bs: & \P(V^*) &\rightarrow &\C\P^{N}\\
           &  [ \mathrm{ev}]  & \mapsto & [ \mathrm{ev}(s_\alpha)]
\end{array}
\end{equation}
and thus an embedding $f_\bs:=\Phi_\bs\circ \iota$ of $X$ in $\C\P^{N}$, where $N+1=\dim(V)$.
In this section, we will consider a moduli problem for embedded submanifolds of $\C\P^N$
with same degree and dimension as $X_\bs:=f_\bs(X)$.
\subsection{Chow stability}
\noindent The idea behind Chow points and schemes is to replace the embedded submanifold $X_\bs$ by a hypersurface in a higher dimensional projective space.
 Consider $n+1$ copies of the dual projective space $\C\P^{N*}\times \ldots \times \C\P^{N*}$ and the divisor:
 \begin{equation*}
  \label{eq:divisor of X}
  D_{X_\bs}:=\lbrace (H_0,\cdots,H_n)\in \C\P^{N*}\times \ldots \times \C\P^{N*}\:\vert\: H_0\cap \ldots \cap H_n\cap X_\bs\neq \emptyset \rbrace.
 \end{equation*}
 Set 
 $$
 W:=(\Sym^{d}(\C^{N+1}))^{\otimes n+1}.
 $$
 The divisor $D_{X_\bs}$ is defined by an element $\widehat{\Ch}(X_\bs)$ in $W$, up to a constant.
 The corresponding point $\Ch(X_\bs)\in \P(W) $ is called the
 Chow point of $X_\bs$. A remarkable fact is that $X_\bs$ is entirely characterized by $\Ch(X_\bs)$.
Set $d=\int_X c_1^n(L)$ the degree of $X_\bs$ in $\C\P^{N}$.
Then, consider the space of $n$-dimensional subvarieties of $\C\P^{N}$ of degree $d$:
\begin{equation*}
 \label{def:Chow scheme}
 \cCh_{\P^{N}}(n,d):=\lbrace  Y \hookrightarrow \C\P^{N}\: \vert\: \dim(Y)=n,\: \mathrm{degree}(Y)=d  \rbrace.
\end{equation*}
We then have an injective map:
\begin{equation}
 \label{eq:Chow map}
 \begin{array}{cccc}
  \Ch: & \cCh_{\P^{N}}(n,d) & \to & \P(W) \\
       &  (Y\hookrightarrow \C\P^{N}) & \mapsto & \Ch(Y).
 \end{array}
\end{equation}
and the Chow scheme $\cCh_N(n,d)$ is by definition the image of the $\Ch$ map:
$$
\cCh_N(n,d):=\Ch(\cCh_{\P^{N}}(n,d)).
$$
By construction, $\Ch$ is a $1:1$ correspondence between $\cCh_{\P^{N}}(n,d)$ and the Chow scheme $\cCh_N(n,d)$,
and the identification is often implicit in the literature.
\begin{remark}
 In this paper, we restrict ourselves to smooth subvarieties of $\C\P^N$. For moduli considerations,
 one may consider subschemes as well.
\end{remark}
\noindent The $\mathrm{SL}_{N+1}(\C)$-action on $\C^{N+1}$ induces an action on $W$,
and we are interested in the GIT quotient of the Chow scheme under this action.
To simplify notation, set $G^c=\mathrm{SL}_{N+1}(\C)$.
We will restrict ourselves to a single $G^c$-orbit:
$$
\cZ:=G^c\cdot \Ch(X_\bs)
$$
Recall the definition:
\begin{definition}
 The polarized manifold $(X,L)$ is Chow polystable if the $G^c$-orbit of $\widehat\Ch(X_\bs)$ is closed in $W$
 for any $\bs\in\cB(V)$.
\end{definition}
\noindent Note that this definition is independent on the choice of $\bs\in\cB(V)$, as seen below.
We can assume, up to replacing $L$ by a high tensor power, that $\Aut_0(X,L)$ acts on $L$
and thus on $V$ (see e.g. \cite{kob}).
We then consider the induced group representation
$$
	\rho: \mathrm{Aut}_{0}(X,L) \to  \mathrm{SL}(V).
$$
The natural right action of $G^c$ on $\cB(V)$ commutes with the left action of $\C^* \times \Aut_0(X)$ on
$\cB(V)$, where $\C^*$ acts by scalar multiplication.
The space $\cZ$ is then identified with the quotient space
$$
\cZ\simeq \cB(V) / (\C^* \times \Aut_0(X))
$$
via the $\Ch$ map. 
The bundle $\cO_{\P(W)}(1)$ restricts to an ample line bundle $\cL$ on $\cZ$ and $(\cZ,\cL)$ is a smooth polarized K\"ahler manifold.
In the next section, we recall the moment map condition for Chow polystability.

\subsection{Balanced embeddings}
\label{sec:balanced embeddings}
We briefly recall the symplectic aspects related to Chow stability. For more details and proofs we refer to \cite{zha,luo,PS03,ps04}.
Denote by $G=\mathrm{SU}(N+1)$ the compact form of $G^c$.
The line bundle $\cL$ carries a hermitian metric $h$, called Chow metric, whose curvature $\Om$
is a $G$-invariant K\"ahler form on $\cZ$ and such that the $G$-action on $(\cZ,\Om)$ is hamiltonian.
Denote by $\g=\Lie(G)$ the Lie algebra of $G$, identified to its dual 
with the invariant non-degenerate pairing 
$$<\xi,\eta>=\trace(\xi\cdot \eta^*).$$
Denote also by $m_0$ the moment map for the $G$-action on $(\C\P^N,\om_{FS})$:
\begin{equation}
 \label{mmap CPN}
 \begin{array}{cccc}
  m_0 : & \C\P^N &\to &\g \\
               &   \z  & \mapsto & i( \frac{\z^*\cdot \z}{\vert \z \vert^2} )_0
 \end{array}
\end{equation}
where the subscript $0$ stands for the trace-free part.
Then the moment map $\mu_0$ for the $G$-action on $(\cZ,\Om_\cZ)$ reads:
\begin{equation}
 \label{eq:mmap}
 \begin{array}{cccc}
  \mu_0 :& \cZ & \to & \g \\
       &  \bs & \mapsto & \int_{X_\bs}\: m_0\: \om_{FS}^n.
 \end{array}
\end{equation}
\begin{definition}
 A zero of $\mu_0$ is called a balanced embedding.
\end{definition}
\noindent From \cite{zha,luo}, the associated GIT notion is Chow polystability:
\begin{theorem}
 The manifold $(X,L)$ is Chow polystable if and only if it admits a balanced embedding.
\end{theorem}
\noindent If $\Aut_0(X,L)$ is not discrete, there is a caracter on its Lie algebra whose vanishing is a necessary condition to Chow
polystability of $(X,L)$. To introduce this character, we first express the $G^c$-action on $\cL$ by mean of the moment map.
For any $\xi\in\g^c$, denote by $X_\xi$ the infinitesimal action of $\xi$ on $\cZ$. Then the lift of
$X_\xi$ is given by
\begin{equation}
 \label{eq:lifting action}
 \widehat{X_\xi}= X^h_\xi + 2\pi < \mu_0,\xi >  X_1,
\end{equation}
where $X^h_\xi$ is the horizontal lift with respect to the Chern connection of $(\cL,h)$ and where
$X_1$ denotes the vector field generating the $U(1)$ action on the fibers of $(\cL,h)$.
Let $\bs\in \cZ$, $G_\bs$ be its stabiliser in $G$ and $\g_\bs=\Lie(G_\bs)$. Define the map $\cF$:
\begin{equation}
\label{eq:obstruction chow stability}
\begin{array}{cccc}
\cF : & \g_\bs^c & \to & \C\\
&            \xi  & \mapsto & < \mu_0(\bs) , \xi >.
\end{array}
\end{equation}
We refer to \cite{wa} for a proof of the folowing:
\begin{proposition}
For any $\xi\in\g_\bs$, the function 
$$
G^c\ni g \mapsto < \mu_0(g\cdot \bs) , \Ad_g \xi >
$$ is constant.
Moreover $\cF$ defines a character on $\g_\bs^c\simeq \Lie(\Aut_0(X,L))$.
\end{proposition}
\begin{remark}
 The character $\cF$ is a quantization of the celebrated Futaki character \cite{fut}.
\end{remark}
\noindent By definition, the existence of a balanced embedding imposes the vanishing of $\cF$.
For any $\xi\in \g_\bs$, $X_\xi$ and its lift $X_\xi^h$ vanish at $\bs$. However,
if $<\mu_0(\bs) , \xi>$ is not zero, the orbit of a lift $x\in\cL_\bs$ under the $G_\bs^c$-action will contain zero,
and the $G^c$-orbit will not be closed. To overcome this issue, following Sz\'ekelyhidi \cite{sz},
one considers a relative notion of stability. In the next section, we introduce a relative Chow scheme
and the associated notion of stability.

\section{Relative Chow stability}
\label{sec:rel Chow}

\noindent In the following, we will restrict ourselves to embedded submanifolds of $\C\P^{N}$ that are invariant under a specified torus of symmetries.

\subsection{Relative embeddings}
\label{sec:rel embed}
Fix a maximal torus $T$ in $\mathrm{Aut}_{0}(X,L)$, and denote its complexification $T^{c}$.
Recall that we have a representation
$$
\rho : \Aut_0(X,L) \to \mathrm{SL}(V).
$$
To simplify notations, we also denote the image of $T^{c}$ under $\rho$ by $T^{c}$.
The action of the complexified torus on
$V$ induces a weight decomposition
$$
V=\bigoplus_{\chi\in w_V(T)} V(\chi)
$$
where $w_V(T)$ is the space of weights for this action.
Let $N^\chi$ be the dimension of $V(\chi)$.
We consider the space of basis adapted to this decomposition:
$$
\cB^T(V):=\left \{ ( s_i^{\chi} )_{\chi \in w(T); i=1..N^\chi} \in (V)^{N+1} \vert \det(s_i^\chi)\neq 0 \text{ and } \forall (\chi, i),\; s_i^\chi\in V(\chi)  \right \}.
$$
Assume now that the embedding $f_\bs=\Phi_\bs\circ\iota$ is given by some $\bs\in\cB^T(V)$. The representation $\rho$ induces via the choice of basis
$\bs$ a representation of $T^c$ on $\C^{N+1}$:
$$
\rho : T^c \to \mathrm{SL}(\C^{N+1})
$$
and a weight decomposition
$$
\C^{N+1}=\bigoplus_{\chi\in w_V(T)} \C^{N_\chi}
$$
where $T^c$ acts with weight $\chi$ on $\C^{N_\chi}$. Note that these representations and weight decompositions
do not depend on $\bs\in\cB^T(V)$, and by construction $X_\bs$ is invariant under the $\rho(T^c)$-action.
Consider now the set of subvarieties of $\C\P^N$ of dimension $n$ and degree $d$ that are invariant under the $T^c$-action:
 \begin{equation*}
 \cCh_{\P^{N}}^T(n,d):=\lbrace  Y \hookrightarrow \C\P^{N}\: \vert\: \dim(Y)=n,\: \mathrm{degree}(Y)=d,\: \rho(T^c)\cdot Y=Y  \rbrace.
\end{equation*}
\begin{definition}
 \label{def:relative Chow scheme}
The relative Chow scheme $\cCh_{N}^T(n,d)$ is defined by:
$$
\cCh_{N}^T(n,d):=\Ch(\cCh_{\P^{N}}^T(n,d)).
$$
\end{definition}
\noindent Alternatively, $\rho$ induces a natural representation of $T^c$ on $\Sym^{d}(\C^{N+1})$, and thus on $W=(\Sym^{d}(\C^{N+1}))^{\otimes n+1}$.
 The relative Chow scheme is then given by the invariant Chow points in
 $\P(W)$.
 Denote the weight decomposition on $W$ by
 $$
W=\bigoplus_{\chi\in w_W(T)} W(\chi).
 $$
By construction, for any $Y\in \cCh_{\P^N}^T(n,d)$, the associated divisor $D_Y$ is $T^c$-invariant
for the diagonal action of $T^c$ on $\C\P^{N*}\times\ldots\times\C\P^{N*}$, and thus
there is a unique weight $\chi_Y\in w_W(T)$ such that $\widehat\Ch(Y)\in W(\chi_Y)$.
Thus the relative Chow scheme lies in the disjoint union
$$
\cCh_N^T(n,d)\subset \bigcup_{\chi\in w_W(T)} \P(W(\chi)),
$$
where we identify $\P(W(\chi))$ with its image in $\P(W)$.
\noindent Consider now the centralizer of $\rho(T^c)$ in $SL_N(\C)$:
$$
G_T^c=S(\Pi_\chi GL_{N^\chi}(\C))
$$
which is the complexification of
\begin{equation}
\label{eq:G}
	G_T:=S(\Pi_\chi U(N^\chi)).
\end{equation}
The $G_T^c$-action on $\C\P^N$ induces a $G_T^c$-action on $\cCh_{\P^N}^T(n,d)$. 
We also have an induced action on $W$, under which the $\Ch$ map is $G_T^c$-equivariant.
The orbit of $\Ch(X_\bs)$ under this action 
corresponds to the set of Chow points:
$$
\Ch(\lbrace \Phi_{\bs\cdot g} \circ \iota(X),\: g\in G_T^c \rbrace)=\Ch(\lbrace \Phi_{\bs} \circ \iota(X),\: \bs\in\cB^T(V) \rbrace)
$$
for the natural right action of $G_T^c$ on $\cB^T(V)$. In the next section, we consider the GIT notions associated to
the $G_T^c$ action on $\cCh_{\P^N}^T(n,d)$.

\subsection{Relative stability}
\label{sec:rel stabi}
We come back to the symplectic picture initiated in Section \ref{sec:balanced embeddings}, following the
abstract setting of \cite{sz}.
The obstruction (\ref{eq:obstruction chow stability}) motivates the following: instead of looking for zeros of $\mu_0$,
try to minimize the function
$$
G^c\ni g \mapsto \vert\vert \mu_0(g\cdot \bs)\vert\vert^2.
$$
A critical point for the square norm of the moment map satisfies the Euler-Lagrange condition $\mu_0(\bs)\in \g_\bs$.
Equivalently, $\mu_0(\bs)$ is equal to its orthogonal projection on the stabiliser of $\bs$.
To have a common stabiliser for any $\bs$, we consider $T$ a maximal compact subtorus of $G_\bs$ as in Section \ref{sec:rel embed},
and restrict ourselves to the $G_T^c$ orbit of $\bs$:
$$
\cZ^T:=G_T^c\cdot \bs.
$$
For any $\bs\in\cZ^T$, the stabilizer of $\bs$ in $G_T$ is $T$. 
Denote by $\t$ the Lie algebra of $T$.
We introduce the extremal vector $\mu_T$ to be the orthogonal projection of $\mu_0(\bs)$ onto $\t$.
One obtains the relative notions of balanced embeddings and Chow stability \cite{ma040,ah}:
\begin{definition}
 An embedding $f_\bs$ is balanced relative to $T$ if $\mu_0(\bs)=\mu_T$.
\end{definition}
\noindent Consider at the level of Lie algebras the orthogonal decompositions
$$
 \g_T=\t\oplus \g_{T^\perp} \:\textrm{ and }\:  \g_T^c=\t^c \oplus \g_{T^\perp}^c.
$$
Denote by $G_{T^\perp}$ and $G_{T^\perp}^c$ the connected Lie subgroups of $G^c$ associated to $\g_{T^\perp}$ and $\g_{T^\perp}^c$.
\begin{definition}
\label{def:rel chow stability}
 The polarized manifold $(X,L)$ is Chow polystable relative to $T$ if the $G_{T^\perp}^c$-orbit of $\widehat\Ch(X_\bs)$ is
 closed in $W$ for any $\bs\in\cB^T(V)$.
\end{definition}
\noindent Mabuchi proved the following theorem \cite{ma040}:
\begin{theorem}
The polarized manifold $(X,L)$ is Chow polystable relative to $T$ if and only if it admits a balanced embedding relative to $T$.
\end{theorem}
\noindent Note that the restriction to the $G_{T^\perp}^c$-orbit is crucial. Indeed, as noticed in Section \ref{sec:rel embed},
and by $T^c$-equivariance, there is a unique $\chi_X\in w_W(T)$ such that
the set $\Ch(\lbrace X_\bs,\: \bs\in \cB^T(V)\rbrace)$ lies in $\P(W(\chi_X))$. By construction,
this character corresponds to the exponential of (a multiple) of the character $\cF$ defined in (\ref{eq:obstruction chow stability}).
Thus, if $\cF\neq 0$, the $T^c$ orbit of $\widehat\Ch(X_\bs)$ will contain zero for any $\bs\in\cB^T(V)$ and
its $G_T^c$-orbit will not be closed. 
\begin{definition}
The destabilizing character of $(X,L)$ is the unique $\chi_X\in w_W(T)$ such that $\Ch(X_\bs)\in\P(W(\chi_X))$ for any $\bs\in\cZ^T$.
\end{definition}
\noindent In the following section, we explain how $\sigma$-balanced
embeddings overcome the issue caused by the destabilizing character, while being equivalent to relatively balanced embeddings.

\section{GIT of $\sigma$-balanced metrics}
\label{sec:GIT sigma}
\noindent An alternative notion of relative balanced embeddings has been introduced by Sano \cite{s}.
A $\sigma$-balanced metric is a self-similar solution to Donaldson's dynamical system,
and provides a quantization of extremal metrics \cite{st1}. 

\subsection{Definition of $\sigma$-balanced embeddings}
\label{sec:def sigma balanced}
We refer to \cite{st2} for a detailed treatment of this section.
Let $\sigma\in T^c$. We also denote by $\sigma$ its image in $G^c_T$ under the representation $\rho$.
From the symplectic point of view, the idea behind $\sigma$-balanced embeddings
is to twist the invariant pairing $<\cdot,\cdot>$ 
by $\sigma \in T^c$ to obtain a moment map orthogonal to the stabiliser $\t$.
That way, the obstruction $\cF$ from (\ref{eq:obstruction chow stability}) vanishes.
Equivalently, one can twist the moment map $\mu_0$ and the symplectic form $\Om$.
On the space $\cZ^T$, consider the map:
\begin{equation}
 \label{eq:definition sigma mmap}
 \begin{array}{cccc}
  \mu_0^\sigma : & \cZ^T & \to & \g_T\\
                 & \bs    & \mapsto &  \int_{X_\bs} \: m_0^\sigma\: \om_{FS}^n,
 \end{array}
\end{equation}
where 
\begin{equation*}
 \begin{array}{cccc}
  m_0^\sigma (\z) = i( \frac{(\sigma\cdot \z)^*(\sigma\cdot \z)}{\vert \z \vert^2} )_0.
 \end{array}
\end{equation*}
\noindent There is a $G_T$ invariant K\"ahler form $\Om^\sigma$ on $\cZ^T$ so that 
$\mu^\sigma_0$ is a moment map for the $G_T$ action on $(\cZ^T,\Om^\sigma)$.
\begin{definition}
\label{def:sigma balanced embedding}
 The embedding $f_\bs: X\to \C\P^N$ is called $\sigma$-balanced if $\mu_0^\sigma(\bs)=0$.
\end{definition}
\noindent The following provides a character on $\t^c$, independent on $\bs\in\cZ^T$, whose vanishing is a necessary condition
for the existence of a $\sigma$-balanced embedding:
\begin{equation}
 \label{eq:obstruction sigma balanced}
 \begin{array}{cccc}
  \cF^\sigma : & \t^c & \to & \C \\
               & \xi &  \mapsto & < \mu^\sigma_0(\bs) , \xi >.
 \end{array}
\end{equation}
\noindent This character is actually the linearisation of a convex fonctional $\cG$ on $T^c$:
\begin{equation}
 \label{eq:convexG}
 \begin{array}{cccc}
  \cG : & T^c & \to & \R\\
        & \sigma & \mapsto & \int_{X_\bs} \frac{\vert\sigma \cdot \z\vert^2}{\vert \z \vert^2}\: \om_{FS}^n.
 \end{array}
\end{equation}
\noindent As a corollary, there is a unique $\sigma\in T^c$, up to $T$, so that $\cF^\sigma=0$. This motivates the following definition:
\begin{definition}
 An element $\sigma\in T^c$ such that $\cF^\sigma=0$ is called optimal weight.
\end{definition}
\noindent In the next section, we investigate the relation between optimal weight an extremal vector $\mu_T$, as well as 
the relation between $\sigma$-balanced embeddings and Chow stability.

\subsection{Twisting the torus}
\label{sec:twisting torus}
To relate $\mu$ and $\mu^\sigma$, we need to consider a bigger group action, and include the homothetic translation
on $L$, or $\C^{N+1}$, in the picture. Consider the $\C^*$ action
on $L$ given by rotations on the fibers. It induces an action on $\cB(V)$ that correspond to the
action by homotheties on $\C^{N+1}$.
Set 
$$
\tG^c=\C^*\times G^c,
$$
with Lie algebra isomorphic to $\mathfrak{gl}_{N+1}(\C) $.
Then the moment map for the action of $\tG=S^1\times G$ on $(\C\P^N,\om_{FS})$ is
\begin{equation*}
 \label{eq:full mmap FS}
  \begin{array}{cccc}
  m : & \C\P^N &\to &\mathfrak{u}(N+1) \\
               &   \z  & \mapsto & i( \frac{\z^*\cdot \z}{\vert \z \vert^2} ).
 \end{array}
\end{equation*}
The $\tG^c$ actions descends to $(\cZ,\cL)$, with moment map given at $\bs$ by integration of $m$ over $X_\bs$.
In the relative situation, set 
\begin{itemize}
 \item $\tT=S^1\times T$, maximal compact subgroup of symmetries in $\Aut(L)$,
 \item $\tT^c$ its complexification,
 \item $\tG_T^c=\C^*\times G_T^c$, centraliser of $\tT$ in $\tG^c$,
 \item $\tG_T=S^1\times G_T$, its compact form,
 \item $\tt$, $\tt^c$, $\tg_T$, $\tg_T^c$ the corresponding Lie algebras.
\end{itemize}
The action of these groups descends to an action on $\cZ^T$. Note that for any $\bs\in\cZ^T$, its stabiliser
is now $\tT^c$. The $\tG_T$-action is hamiltonian for $\Om$ with moment map:
\begin{equation}
  \label{eq:full mmap}
  \begin{array}{cccc}
  \mu : & \cZ^T &\to & \tg_T \\
        &   \bs  & \mapsto & \int_{X_\bs}\: m\: \om_{FS}^n.
 \end{array}
\end{equation}
We fix now an optimal weight $\sigma\in T^c$.
Note that we still have the orthogonal decomposition
\begin{equation}
 \label{eq:ortho decomp extended}
 \tg_T= \tt \overset{\perp}{\oplus} \g_{T^\perp}=  1\overset{\perp}{\oplus}\t\overset{\perp}{\oplus}\g_{T^\perp}.
\end{equation}
Setting $\t^\sigma:=\sigma\cdot \t \cdot \sigma^*$, 
the decomposition (\ref{eq:ortho decomp extended}) induces the decomposition into orthogonal subspaces (actually Lie subalgebras, see below):
\begin{equation}
\label{eq:ortho decomp sigma}
\tg_T=(\sigma 1 \sigma^* \oplus \t^\sigma)\overset{\perp}{\oplus} (\sigma^{-1}\cdot \g_{T^\perp} \cdot(\sigma^{-1})^*).
\end{equation}
Fix now $\xi=\sigma\eta\sigma^*\in (\sigma 1 \sigma^* \oplus \t^\sigma)$. 
The quantity $<\mu, \xi>$ is constant on $\cZ^T$. Indeed, if $\xi\in \t^\sigma$,
$<\mu,\xi>=<\mu^\sigma_0 , \eta >$ as $\eta$ is trace-free. Then $<\mu, \xi>=\cF^\sigma(\eta)=0$, as $\sigma$ is chosen optimal.
For $\xi=\sigma 1 \sigma^*$, $<\mu, \xi>=\cG(\sigma)$, also independent on $\bs$.
Thus the differential of $<\mu, \xi>$ vanishes on $\cZ^T$, and $\iota_{X_\xi}\Om=0$ by definition of a moment map.
The non-degeneracy of $\Om$ forces $X_\xi$ to be zero, and $\xi$ belongs to the Lie algebra of the stabiliser
of $\bs\in\cZ^T$ for any $\bs$, that is $\tt$. We have shown that
\begin{equation}
 \label{eq:t is t sigma}
(\sigma 1 \sigma^* \oplus \t^\sigma)= \tt.
\end{equation}
From (\ref{eq:ortho decomp extended}) and (\ref{eq:ortho decomp sigma}) we deduce for any optimal weight $\sigma\in T^c$
the decomposition:
\begin{equation}
 \label{eq:decomp last}
 \tg_T= (\sigma 1 \sigma^* \oplus \t^\sigma)\overset{\perp}{\oplus} \g_{T^\perp}.
\end{equation}
From this we have:
\begin{theorem}
\label{theo:sigma balanced iff balanced rel T}
Let $\sigma$ be an optimal weight and let $\bs\in\cZ^T$. Then $X_\bs$ is balanced relatively to $T$ if and only if it is $\sigma$-balanced.
\end{theorem}
\begin{proof}
 The $\sigma$-balanced condition is equivalent to $\mu_0^\sigma=0$, that is $\sigma\mu\sigma^* =c_\sigma 1$ for some real constant $c_\sigma$.
 By choice of optimal weight, $\forall \xi\in \t$,
$$
\cF^\sigma(\xi)=<\mu^\sigma,\xi>=<\sigma \mu\sigma^*, \xi >=0,
$$
where we use that $\xi$ is trace-free. Thus, $\bs$ is $\sigma$-balanced if and only if $\sigma \mu(\bs)\sigma^*\in (1\oplus \t)$.
From (\ref{eq:t is t sigma}), this is equivalent to $\mu(\bs)\in (1\oplus \t)$,
that is $\bs$ being balanced relative to $T$.
\end{proof}
\noindent As a corollary, we recover Theorem \ref{theo:intro}:
\begin{corollary}
 $(X,L)$ admits a $\sigma$-balanced embedding if and only if it is Chow polystable relative to $T$.
\end{corollary}
\noindent We now relate the choice of optimal weight $\sigma\in T^c$ to the relative Chow scheme. From now on, we shall assume that
the restriction of the pairing $<\cdot,\cdot>$ to $\tt$ is normalized to be rational on the kernel of the exponential map.
\begin{lemma}
\label{lem:sigma rationality}
 For $\sigma$ an optimal weight, $(\t^\sigma)^c$ is the Lie algebra of an algebraic torus in $\tG_T^c$.
\end{lemma}
\begin{proof}
 It is clear that $\t^\sigma$ is an abelian Lie algebra, and generates a compact analytic subgroup of $\tG_T$.
 Remains to show that it generates an algebraic group.
 The projection of $\mu$ to $\tt$ doesn't depend on $\bs$, and equals
 $c 1 + \mu_T$ for a constant $c$. By \cite[Lemma 3.3]{sz}, $c 1+\mu_T$ generates an algebraic subgroup of $\tG_T$.
 Note that the proof of \cite[Lemma 3.3]{sz} can be adapted directly to our context, without assuming
 the existence of a solution to $\mu(\bs)=c 1 + \mu_T$, because $\tT$ stabilizes all elements $\bs\in\cZ^T$.
 But then $\sigma^{-1} c_\sigma 1 (\sigma^{-1})^*=c1 + \mu_T$ 
 (see proof of Theorem \ref{theo:sigma balanced iff balanced rel T}), so that $\sigma^{-1} c_\sigma 1 (\sigma^{-1})^*$ also generates
 an algebraic subgroup. We conclude that $\sigma (c_\sigma^{-1}) 1 \sigma^*$ satisfies the rationality condition that ensures that
 $\t^\sigma ( = c_\sigma^{-1}\cdot \t^\sigma )$ generates an algebraic torus in $\tG_T^c$.
\end{proof}
\noindent Let's denote by $T^\sigma$ the algebraic torus in $\tG^c_T$ generated by $\t^\sigma$.
Note that any $X_\bs$, $\bs\in \cZ^T$, is $T^\sigma$-invariant, so that 
$\cZ^T$ embeds into the relative Chow scheme $\cCh^{T^\sigma}_N(n,d)$ via the $\Ch$ map.
As in Section \ref{sec:rel Chow}, there is a weight decomposition under the $T^\sigma$-action:
$$
W=(\Sym^{d}(\C^{N+1}))^{\otimes n+1}=\bigoplus_{\chi\in w_W(T^\sigma)} W(\chi)
$$
and $\cZ^T$ is embedded in $\P(W(\chi'))$ for a unique $\chi'\in w_W(T^\sigma)$.
\begin{lemma}
 The character $\chi'$ is trivial.
\end{lemma}
\begin{proof}
As $\sigma$ is an optimal weight, $\mu$ is orthogonal to the Lie algebra $\t^\sigma$.
From (\ref{eq:lifting action}), we see that the linearisation of the action of $T^\sigma$ is trivial,
that is $\chi'=0$.
\end{proof} 
\noindent Let $G_\sigma^c$ be the connected Lie subgroup of $\tG_T^c$ generated by $(\t^\sigma)^c\oplus \g_{T^\perp}^c$.
We then have shown:
\begin{theorem}
 The manifold $(X,L)$ admits a $\sigma$-balanced embedding if and only if the $G_\sigma^c$-orbit of $\widehat \Ch(X_\bs)$
 is closed for any $\bs\in \cB^T(V)$.
\end{theorem}
\noindent The optimal weight provide a relative setting where the only destabilizing symmetry
is generated by $\sigma^{-1} 1 (\sigma^{-1})^*$.
\subsection{Optimal weight and destabilizing character}
\label{sec:optimal weight}
Recall from Section \ref{sec:rel embed} that there is a unique $\chi_X\in w_W(T)$ such that
the set $\Ch(\lbrace X_\bs,\: \bs\in \cB^T(V)\rbrace$ lies in $\P(W(\chi_X))$. In this short section, we give the precise relation
between $\chi_X$ and optimal weight $\sigma$. For this, we will restrict our attention to the $\tT^c$ action on $W$.

\noindent First notice that there is a $1 - 1$ correspondence between the weights $w_W(\tT)$ of the $\tT^c$-action
and the weights of the $T^c$-action on $W$ given by
\begin{equation*}
 \begin{array}{ccc}
  w_W(T) & \rightarrow & w_W(\tT) \\
   \chi  &  \mapsto    &  \chi': (\lambda, t) \mapsto \lambda^{d(n+1)}\cdot\chi(t)
 \end{array}
\end{equation*}
where $(\lambda,t)\in \C^*\times T^c$.
In particular, the space $W$ decomposes into the same direct sum of invariant subspaces under the action of these two tori,
and $\Ch(X_\bs)\in \P(W(\chi_X'))$ for any $\bs\in\cB^T(V)$.

\noindent Let $\sigma$ be an optimal weight and set 
$$
\alpha^\sigma:= i \sigma^{-1} \cdot \cG(\sigma) \Id \cdot (\sigma^{-1})^*.
$$
\begin{proposition}
\label{prop:relation weight character}
The relation between the optimal weight $\sigma$ and the destabilizing character of $(X,L)$ is
given by $$\alpha^\sigma=\chi'_X.$$
\end{proposition}
\begin{proof}
From the proof of Lemma \ref{lem:sigma rationality}, $\alpha^\sigma$ is rationnal, and thus can be written
$$
\alpha^\sigma = \frac{\chi^\sigma}{n^\sigma},
$$
for $\chi^\sigma$ a character in $Hom(\tT^c,\C^*)$ (by using derivation, we identify the characters and elements in $(\tt^c)^*$).
Let $\C^{\sigma}$ be the trivial line bundle over $\cZ^T$, with $\tT^c$-linearisation given by $(\chi^\sigma)^{-1}$.
Recall from the discussion of Section \ref{sec:twisting torus} that $\alpha^\sigma$ is the orthogonal projection $\pi_{\tt}\mu$
of $\mu$ to $\tt$, independently on $\bs\in\cZ^T$.  Then by construction, 
the moment map for the $\tT^c$-action on $(\cZ^T,\cL^{n^\sigma}\otimes\C^{\sigma})$ is nothing but 
$$
n^\sigma \pi_{\tt} \mu - \chi^\sigma=0.
$$
Thus the points in $\cZ^T$ are semi-stable with respect to the $\tT^c$-action on $\cL^{n^\sigma}\otimes\C^\sigma$.
This action must then be trivial. But $\tT^c$ acts on $\cL^{n^\sigma}\otimes\C^\sigma$ with weight
$$(\chi'_X)^{n^\sigma}\cdot(\chi^\sigma)^{-1}.$$
We conclude that $\alpha^\sigma=\chi'_X$.
\end{proof}
\begin{remark}
 Note that from Proposition \ref{prop:relation weight character},  the extremal vector $\alpha^\sigma$
 only depends on the connected component of $\cCh^T_{N}(n,d)$ where the Chow line of $(X,L)$ lies.
 In particular, it is invariant under embedded $T$-equivariant complex deformations of $(X,L)$.
 Considering tensor powers $L^k$ of $L$, and letting $k$ go to infinity, 
 the extremal vectors $\alpha^{\sigma_k}$ converge to the extremal vector field of $(X,L)$ with respect to $T$ (see \cite{st2}).
 One recovers the fact that the extremal vector field is invariant under polarized $T$-equivariant complex deformations \cite{lej}.
\end{remark}

\end{document}